\documentclass[11pt]{amsart}
\usepackage{pxfonts}
\usepackage{amsmath,amsfonts,amscd,amssymb,epsfig,euscript,bm}
\usepackage{amsxtra,mathrsfs}
\usepackage{pdfsync}
\newtheorem{theorem}{Theorem}[section]

\newtheorem{lemma}{Lemma}[section]

\theoremstyle{definition}
{}

\theoremstyle{remark} 
\newtheorem{remark}{Remark}[section]

\newcommand{\del}{\partial}

\newcommand{\delb}{\bar{\partial}}
\newcommand{\C}{{\mathbb{C}}}

\newcommand{\cF}{{\mathcal{F}}}
\newcommand{\pbp}{\mathcal{\partial \bar{\partial}}}
\newcommand{\spbp}{\mathcal{\sqrt{-1}\partial \bar{\partial}}}
\DeclareMathOperator{\Tr}{Tr}
\numberwithin{equation}{section}
\usepackage{hyperref}
\hypersetup{colorlinks,citecolor=blue,plainpages=false,hypertexnames=false}
\begin{document}
\title[A note on higher extremal metrics]{A note on higher extremal metrics}
\author{Vamsi Pritham Pingali}
\address{Department of Mathematics, Indian Institute of Science, Bangalore, India - 560012}
\email{vamsipingali@iisc.ac.in}
\begin{abstract}
In this paper we introduce ``higher extremal K\"ahler" metrics. We provide an example of the same on a minimal ruled surface. We also prove a perturbation result that implies that there are non-trivial examples of ``higher constant scalar curvature" metrics, which are basically metrics where the top Chern form is harmonic. We also give a relatively short proof of Liu's formula for the Bando-Futaki invariants (which are obstructions for the existence of harmonic Chern forms) of hypersurfaces of projective space. 
\end{abstract}
\maketitle
\section{Introduction}
The problem of finding K\"ahler-Einstein metrics and more generally, extremal K\"ahler metrics, is of active interest (for instance see \cite{Sz} and the references therein). Extremal metrics may be characterised as K\"ahler metrics for which the gradient of the scalar curvature (expressed as $S=\frac{n c_1 \wedge \omega ^{n-1}}{\omega^n}$) is a holomorphic vector field. Special cases of these are the constant scalar curvature K\"ahler (cscK)  metrics which we interpret as those metrics for which the first Chern form is harmonic \cite{Ban}.  \\
\indent The Chern classes are important objects in algebraic geometry. In addition to the classes, the first Chern-Weil form itself is quite natural to study because it is the Ricci form for a K\"ahler manifold. Indeed, the first Chern form was used by Yau to prove the Bogomolov-Miyaoka-Yau inequality as a consequence of the Calabi conjecture \cite{Yaucal}.  As Yau stated in \cite{Yaupers}, the higher Chern-Weil forms are quite mysterious. That being said, we note that at the level of classes the top Chern class is the Euler class. Therefore, studying the top Chern form might potentially lead to interesting consequences. We are thus led to study the following equation :
\begin{gather}
c_n (\omega) = \lambda \omega ^n
\label{maineq}
\end{gather}
where the gradient of $\lambda$ is a holomorphic vector field. We call these metrics \emph{higher extremal K\"ahler} and if $\lambda$ is a constant, i.e., the top Chern form is harmonic, then we dub them as \emph{higher constant scalar curvature (hcscK)}. \\
\indent The hcscK metrics and their avatars were considered earlier by Bando \cite{Ban} who came up with an obstruction for their existence. Another version of the higher extremal metrics was studied by Futaki \cite{Futp, Futpe} where he considered the perturbed scalar curvature $S(J,t) = \frac{c_1 + tc_2 + t^2 c_3 +\ldots}{\omega^n}$ where $t$ is a small real number. Our question is the case for large $t$ in a sense. So Futaki's results do not apply in any direct manner that the author can see.\\
\indent In this paper we study examples of higher extremal and hcscK metrics. Our first example comes from a minimal ruled surface. For the usual extremal K\"ahler metrics this example was first studied in \cite{TF} and more general results were proven in \cite{ACGT}. 
\begin{theorem}
Let $X$ be $\mathbb{P}(L\oplus \mathcal{O})$ where $L$ is a degree $-1$ line bundle over a genus $2$ surface $\Sigma$. Let $C$ be the Poincar\'e dual of any fibre and $S_{\infty}$ be the copy of $\Sigma$ corresponding to the line $L\oplus \{ 0 \}$. There exists a K\"ahler metric $\omega$ in the class $2\pi(C+S_{\infty})$ such that 
\begin{gather}
c_2 (\omega) = \frac{\lambda}{2(2\pi)^2} \omega^2, 
\end{gather}
where $\nabla ^{(1,0)} \lambda$ is a holomorphic non-zero vector field on $X$, i.e., it is higher extremal K\"ahler but not hcscK. 
\label{mainrule}
\end{theorem}
\begin{remark}
The aforementioned theorem does not assert that for $2\pi(C+m S_{\infty})$ where $m>1$ there are no extremal metrics. The author suspects that there might be a maximum $m$ (just as in the usual extremal K\"ahler case) beyond which there may not exist a solution. The proof is by reducing the equation to an ODE \footnote{It is a version of Chini's equation.} that unfortunately is not integrable and is non-autonomous. The analysis of the ODE is somewhat delicate. In contrast to the usual case \cite{TF} where the corresponding ODE always has a solution satisfying the desired boundary conditions (but it is not clear that the solution actually gives rise to a K\"ahler metric) in our case the difficulty lies with the existence of a solution to the ODE satisfying the boundary conditions. Also, the proof of theorem \ref{mainrule} shows that the assertion of the higher extremal metric not being hcscK is true regardless of $m$.
\label{afterthm}
\end{remark}
\indent In our quest to find more examples, we note that the hermitian symmetric spaces are hcscK. This is because their metric, curvature, and hence characteristic forms are constant linear combinations of invariant differential forms. Actually, in the case of a surface $X$ with ample canonical bundle Yau \cite{Yaucal} showed that if  $c_1^2=3c_2$ numerically, then indeed it admits hcscK metrics and that they are all K\"ahler-Einstein as well (by virtue of them being ball quotients). \\
\indent It is natural to wonder if there are non-trivial (i.e. not $X_1 \times X_2$ with product K\"ahler-Einstein metrics) examples of hcscK metrics. Also, near the symmetric K\"ahler-Einstein metrics are there any other hcscK metrics, i.e., does local uniqueness hold ? The following perturbation result addresses these questions in some cases.
\begin{theorem}
Suppose that either $(X,\omega)=(\mathbb{D}^1/\Gamma_1 \times \mathbb{D}^1/\Gamma_2,\pi_1^{*}\omega_1+\pi_2^{*}\omega_2)$ where $\omega_1, \omega_2$ are constant curvature metrics or $X=\mathbb{D}^2/\Gamma$ equipped with a metric $\omega$ of constant holomorphic sectional curvature. Suppose $\tilde{\omega}$ is any closed real $(1,1)$-form. There exists an $\epsilon_1, \epsilon_2>0$ such that for $\vert t \vert < \epsilon_1$ there exists a unique smooth function $\phi$ of zero average (with respect to $\omega$) depending smoothly on $t$ satisfying $\Vert \phi \Vert_{C^{4,\alpha}} < \epsilon_2$ such that $\omega + t \tilde{\omega} + \spbp \phi$ is hcscK.
\label{mainperturb}
\end{theorem}
\begin{remark}
Consider surfaces of general type satisfying $c_1^2=3c_2$. Noether's formula, Hodge theory, and the fact that $c_2$ is the Euler characteristic allows us to prove that $h^{1,1}=h^{2,0}+h^{1,0}+1>1$ for any such surface (of which there are infinitely many \cite{Borel}) other than the $100$ fake projective planes \cite{CS1,prasad}. For instance, the Cartwright-Steger surface \cite{CS} is a concrete example. For such surfaces one can come up with non-trivial examples of $\tilde{\omega}$ and hence by theorem \ref{mainperturb} find non-trivial hcscK metrics that are not K\"ahler-Einstein.
\label{ex} 
\end{remark}
\indent As pointed out earlier, whenever holomorphic vector fields exist the Bando-Futaki invariant provides an obstruction for the existence of hcscK metrics. It has been computed explicitly in very few cases, most notably by Liu \cite{Liu} for hypersurfaces in $\mathbb{CP}^n$. Her formula can be used to come up with examples of non-existence of hcscK metrics. The theorem we are alluding to is theorem $(1.1)$ of \cite{Liu} :
\begin{theorem}[Liu]
Let $M$ be a hypersurface in $\mathbb{CP}^n$ defined by a homogeneous polynomial $F$ of degree $d \leq n$. Let $Y$ be a holomorphic vector field on $\mathbb{CP}^n$ such that $YF=\kappa F$ for a constant $\kappa$. Then the $q$-th Bando-Futaki Invariant is 
\begin{gather}
\cF_q (Y, \omega _{FS}) = -(n+1-d)^{n-q} \frac{(d-1)(n+1)}{n} \sum _{j=0}^{q-1} (-d)^j (j+1) \binom {n} {q-j-1} \kappa \nonumber 
\end{gather}
\label{Liuthm}
\end{theorem}
In this paper we give a simplified proof of Liu's formula (whilst adhering to her basic strategy). The technique of computation (relying on generating series) might potentially be useful in calculating Bando-Futaki invariants in other cases. The crucial simplification comes from a linear algebra lemma (lemma \ref{Algebra}) that was used to similar effect in \cite{PiTa}. \\
\indent It is interesting to see if Lebrun-Simanca kind of deformation results can be proven for these objects. We hope to explore this and other questions in later works.\\
\emph{Acknowledgments} : The author is sincerely grateful to M.S. Narasimhan and Harish Seshadri for fruitful discussions. Thanks is also in order to two users of mathoverflow.net \cite{mathov} for suggesting the examples in remark \ref{ex}. The author was supported by an SERB grant No. ECR/2016/001356 and also thanks IISc for the Infosys young investigator award. 
\section{A higher extremal metric on a ruled surface}
\indent First we give a high level overview of this section. The aim is to produce a higher extremal metric on a manifold with a lot of symmetry. Akin to \cite{TF} an ansatz reduces the problem to finding a parameter $C$ and solving an ODE depending on $C$ for a function $\phi$ on $[1,m+1]$ where $m$ is a given integer (that specifies the K\"ahler class under consideration) satisfying $\phi(m+1)=0$. The ODE being non-integrable poses difficulties with regard to existence. It turns out that for a connected set of $C$ the ODE does have a smooth solution depending smoothly on $C$ but it is not clear whether $\phi(m+1)=0$. So we produce a value of $C$ so that $\phi(m+1)>0$ and another value so that $\phi(m+1)<0$. Thus there is some admissible $C$ for which $\phi(m+1)=0$. In our proof we can do everything with the exception of producing a $C$ so that $\phi(m+1)<0$. We can do this rigorously only for $m=1$. However, numerically solving the ODE using the Runge-Kutta method on Wolfram Alpha seems to suggest that this is true for higher values of $m$ too. It is just that one does not know explicit error bounds on the numerical solution and hence cannot ``trust" it for a proof. With this bird's eye view in mind we proceed further.\\
\indent Let $(\Sigma, \omega_{\Sigma})$ be a genus $2$ Riemann surface equipped with a metric of constant scalar curvature $-2$. Let $L$ be a degree $-1$ holomorphic line bundle on $\Sigma$ equipped with a metric $h$ such that $-\omega_{\Sigma}$ is the curvature of $h$. Let $X$ be the ruled surface $\mathbb{P}(L\oplus \mathcal{O})$. Just as in \cite{TF, ACGT, HS} we will construct extremal K\"ahler metrics on $X$. In whatever follows we follow the exposition of Sz\'ekelyhidi \cite{Sz}. \\
\indent The strategy is to first consider an ansatz on the total space of $L$ minus the zero section and then extend the resulting metric to all of $X$. One way to potentially produce a metric is to pullback $L$ to its total space and add the curvature of the resulting bundle to the pullback of $\omega_{\Sigma}$. Motivated by this observation one writes the following ansatz. (Let $p: X \rightarrow \Sigma$ be the projection map, $z$ be a coordinate on $\Sigma$, and $w$ be a coordinate on the fibres $L$.)
\begin{gather}
\omega = p^{*} \omega_{\Sigma} +\spbp f(s),
\label{ansatz}
\end{gather}
where $s=\ln \vert(z,w)\vert_h ^2 = \ln \vert w \vert^2 + \ln h(z) $ and $f$ is a strictly convex function that makes $\omega$ a metric. We choose coordinates $(z_0,w_0)$ around a point $Q$ such that $dh (z_0) = 0$. Therefore at $Q$ we have the following equalities.
\begin{gather}
\del s (Q)= \frac{dw}{w} \ , \ \delb s (Q)= \frac{d\bar{w}}{\bar{w}} \nonumber \\ 
\spbp s (Q) = p^{*} \omega _{\Sigma} \nonumber \\
\omega (Q) = (1+f^{'}(s))p^{*} \omega _{\Sigma} + f^{''}(s) \sqrt{-1}\frac{dw \wedge d\bar{w}} {\vert w \vert^2}. 
\label{usefulatQ}
\end{gather} 
The last equation is easily seen to hold at points other than $Q$ as well.\\
 \indent Proceeding to study the K\"ahler class of $\omega$ we see that by the Leray-Hirsch theorem $H^2(X, \mathbb{R}) = \mathbb{R} C \oplus \mathbb{R} S_{\infty}$, where $C$ is the Poincar\'e dual of a fibre (i.e. $C$ is a sphere) and $S_{\infty}$ is a copy of $\Sigma$ sitting in $X$ as the ``infinity section", i.e. the line $L\oplus \{ 0 \}$. It is clear that $C.C =0, C.S_{\infty} = 1 = S_{\infty}. S_{\infty}$. We wish our ansatz to be in the cohomology class $[\omega] = 2\pi (C +mS_{\infty})$ where $m$ is a positive integer. Therefore $[\omega].C = 2\pi m$ and $[\omega].S_{\infty} = 2\pi(1+m)$. Indeed,
\begin{gather}
\displaystyle \int _C \omega = \int _{\mathbb{C} - \{ 0 \}} f^{''}(s)  \sqrt{-1}\frac{dw \wedge d\bar{w}} {\vert w \vert^2}  = 2\pi (\lim_{s\rightarrow \infty} f^{'}(s) - \lim_{s\rightarrow -\infty} f^{'}(s)) = 2\pi m \nonumber \\
 \displaystyle \int _{S_{\infty}} \omega = \int _{\Sigma} \lim_{s \rightarrow \infty} (1+f^{'}(s)) \omega_{\Sigma} = (1+m) \int _{\Sigma} \omega_{\Sigma} = 2\pi(1+m). 
\label{cohomologyeqns} 
\end{gather}
Thus $0\leq f^{'}(s) \leq m$.\\
\indent Returning back to the metric $\omega$ we see that 
\begin{gather}
\omega^2 = 2(1+f^{'}(s))f^{''}(s)  p^{*} \omega _{\Sigma}  \sqrt{-1}\frac{dw \wedge d\bar{w}} {\vert w \vert^2}. 
\label{volumeform}
\end{gather} 
Calculating the curvature matrix of forms $\Theta = \delb (h^{-1} \del h)$ we obtain the following.
\begin{gather}
\Theta = \left [ \begin{array}{cc} -\pbp \ln(1+f^{'}(s)) -2p^{*}\omega_{\Sigma}  & 0 \\ 0 &  -\pbp \ln (f^{''}(s))\end{array} \right ]
\end{gather}
At this point we appeal to the unreasonable effectiveness of the Legendre transform and define 
\begin{gather}
\tau = f^{'}(s) \ , \ f(s) + F(\tau) = s\tau 
\label{deftau}
\end{gather}
Therefore, $s=F^{'}(\tau), \frac{ds}{d\tau} = F^{''}(\tau)$. Since $f^{''}(s)$ seems to crop up often, define (as Hwang-Singer did in \cite{HS}) the so-called momentum profile $\phi(\tau) = f^{''}(s) = \frac{1}{F^{''}(\tau)}$.  Hence $\frac{d\tau}{ds} = \frac{1}{F^{''}(\tau)} = \phi (\tau)$. Moreover, $f^{'''}(s) = \frac{df^{''}(s)}{d\tau} \phi (\tau) = \phi ^{'} \phi$. \\
\indent In terms of $\gamma = \tau+1 \in [1,m+1]$ the curvature form reads as 
\begin{gather}
\sqrt{-1}\Theta =  \left [ \begin{array}{cc} -\spbp \ln(\gamma) -2p^{*}\omega_{\Sigma}  & 0 \\ 0 &  -\spbp \ln (\phi)\end{array}  
 \right ]\nonumber \\
= \left [ \begin{array}{cc} \sqrt{-1}\frac{\partial \gamma \delb \gamma}{\gamma^2} -\frac{1}{\gamma} \spbp \gamma -2p^{*}\omega_{\Sigma}  & 0 \\ 0 & -\left ( \frac{\phi^{'}}{\phi} \right )^{'}\sqrt{-1}\partial \gamma \delb \gamma -\frac{\phi^{'}}{\phi} \spbp \gamma \end{array}  \right ]\nonumber \\
  =  \left [ \begin{array}{cc} \frac{\phi}{\gamma} \left [ \frac{\phi}{\gamma}-\phi^{'}\right ] \frac{dw d\bar{w}}{\vert w \vert^2}-(\frac{\phi}{\gamma}+2) p^{*} \omega_{\Sigma}  & 0 \\ 0 & -\phi^{''}\phi \sqrt{-1} \frac{dw d\bar{w}}{\vert w\vert^2} - \phi^{'} p^{*} \omega _{\Sigma} \end{array}\right ].\nonumber 
\end{gather}
The top Chern form is $c_2 = \frac{1}{(2\pi)^2} \det(\sqrt{-1} \Theta)$ which is
\begin{gather}
c_2 = \frac{1}{(2\pi)^2} p^{*} \omega_{\Sigma} \frac{\sqrt{-1} dw d\bar{w}}{\vert w \vert^2}\frac{\phi}{\gamma^2} \left ( \gamma (\phi + 2\gamma) \phi^{''} + \phi^{'} (\phi^{'}\gamma-\phi) \right ).
\label{topchern}
\end{gather}
We want 
\begin{gather}
c_2 = \frac{1}{(2\pi)^2}\frac{\lambda}{2} \omega ^2
\label{mainagain}
\end{gather}
 to hold for some $\lambda$ whose gradient is a holomorphic vector field, i.e.,
\begin{gather}
\nabla ^{(1,0)} \lambda = \lambda ^{'} \nabla ^{(1,0)} \tau = \lambda ^{'} w\frac{\partial }{\partial w} \nonumber 
\end{gather}
which is a holomorphic vector field if and only if $\lambda ^{'} $ is a constant, i.e., $\lambda = A \gamma +B$ for some $A$ and $B$. \\
\indent So our equation \ref{mainagain} boils down to an ODE for $\phi (\gamma)$.
\begin{gather}
\gamma (\phi + 2\gamma) \phi^{''} + \phi^{'} (\phi^{'}\gamma-\phi) = (A\gamma+B)\gamma^3 \nonumber \\
\Rightarrow 2\gamma ^2 \phi ^{''} + (\frac{\phi \phi^{'}}{\gamma})^{'} \gamma^2 = (A\gamma+B)\gamma ^3 \nonumber \\
\Rightarrow 2\phi^{'} + \frac{\phi \phi^{'}}{\gamma} = A\frac{\gamma^3}{3} + B\frac{\gamma^2}{2} + C \nonumber \\
\Rightarrow (2\gamma+\phi)\phi^{'} =  A\frac{\gamma^4}{3} + B\frac{\gamma^3}{2} + C\gamma,
\label{odesimple}
\end{gather} 
where $A, B, C$ are constants. It can be easily seen that \cite{Sz} for $\omega$ to extend across the zero and infinity sections the following boundary conditions have to be met by $\phi(\gamma)$.
\begin{gather}
\phi(1) = \phi(m+1) = 0 \nonumber \\
\phi^{'}(1) = -\phi^{'}(m+1) = 1 
\label{boundc}
\end{gather}
So we need to solve \ref{odesimple} for $\phi$ as well as for $A,B,C$ so that the boundary conditions \ref{boundc} are met and $\phi >0 \ \forall \ \gamma \in[1,m+1]$. Unfortunately the form $(2x+y)dy-p(x)dx$ is not closed and hence equation \ref{odesimple} cannot be integrated. Nevertheless, one can still prove theorem \ref{mainrule} for $m=1$. In order to do so we prove the following preliminary result about the ODE \ref{odesimple} with boundary conditions \ref{boundc}.\\
\begin{theorem}
Given a positive integer $m$, consider the following ODE.
\begin{gather}
(2\gamma+\phi)\phi^{'} =  A\frac{\gamma^4}{3} + B\frac{\gamma^3}{2} + C\gamma
\label{folODE}
\end{gather}
with the boundary conditions
\begin{gather}
\phi(1) = \phi(m+1) = 0 \nonumber \\
\phi^{'}(1) = -\phi^{'}(m+1) = 1.
\label{boundcthm}
\end{gather}
If $C<M$ (where $M>2$) then there exist linear functions $A(C), B(C)$ depending on a parameter $C$ and a smooth solution $\phi$ to \ref{folODE} on $[1,m+1]$ depending smoothly on $C$ satisfying all the conditions of \ref{boundcthm} except $\phi(m+1)=0$. There exists a $C<M$ such that $\phi(m+1,C) >0$. Moreover, if there exists a smooth solution satisfying all the boundary conditions, then $\phi>0$ on $[1,m+1]$.
\label{partialode}
\end{theorem}
\begin{proof}
We impose the boundary conditions \ref{boundc} on equation \ref{odesimple} to get the following relations between $A,B,C$.
\begin{gather}
2 = \frac{A}{3}+\frac{B}{2}+C \nonumber \\
-2 = \frac{A(m+1)^3}{3}+\frac{B(m+1)^2}{2}+C \nonumber \\
\Rightarrow A(C) = \frac{3C}{m} \left [ 1-\frac{1}{(m+1)^2} \right ] - \frac{6}{m} \left [ \frac{1}{(m+1)^2} +1 \right] \nonumber \\
B(C) = -2C \left [1+\frac{1}{m}-\frac{1}{m(m+1)^2} \right ]+4+\frac{4}{m}\left [1+\frac{1}{(m+1)^2} \right]
\label{relationabc} 
\end{gather}
Thus $A(C)$ and $B(C)$ are linear functions of $C$. Moreover, given $C$, if we manage to solve \ref{odesimple} on $[1,m+1]$ with the initial condition $\phi(1)=0$ then \ref{relationabc} imply that $\phi^{'} =1$ and if we further ensure that $\phi(m+1) = 0$ then $\phi^{'}(m+1) = -1$ automatically. The bottom line is that we have to prove that given $C$, a smooth positive solution depending smoothly on $C$ exists to the initial value problem 
\begin{gather}
 \phi^{'} = \frac{ A(C)\frac{\gamma^4}{3} + B(C)\frac{\gamma^3}{2} + C\gamma}{2\gamma+\phi} \ on \ [1,m+1]  \nonumber \\
\phi(1) =0
\label{initval}
\end{gather}
and that there exists a $C=C_m$ such that $\phi(m+1) = 0$.\\
\indent Near $\gamma=1$ since the right-hand side of \ref{initval} is locally Lipschitz we have a unique smooth solution locally. At this point it is convenient to change variables. Let $v=\frac{(2\gamma+\phi)^2}{2}$. Equation \ref{initval} turns into the following.
\begin{gather}
v^{'} = 2\sqrt{2} \sqrt{v} +p(\gamma) \gamma \nonumber \\
v(1) = 2
\label{vsub}
\end{gather}
We want to find a smooth solution of \ref{vsub} on $[1,m+1]$ so that $v(m+1)=2(m+1)^2$ and $v(\gamma) >2\gamma ^2$ on $(1,m+1)$. \\
\indent As before we have a unique smooth solution depending smoothly on parameters near $\gamma=1$. If there is a solution on $[1,\gamma_*)$ such that $M\geq v\geq\epsilon >0$ then since the right-hand side is $C^1$, by standard ODE theory the solution can be continued past $\gamma_{*}$. An easy comparison argument using $\sqrt{v} \leq kv$ and Gronwall's inequality shows that $v$ is always bounded above. In order to prove lower bounds on $v$ we need to study $p(\gamma)$. 
\begin{lemma}
Let $m\geq 1$ be a given positive integer and $C$ be a real number. The polynomial $p(\gamma) = A(C)\frac{\gamma^3}{3} + B(C)\frac{\gamma^2}{2} + C$  (and hence $p(\gamma) \gamma$) satisfying  $p(m+1) =-2$ and $p(1) =2$ has exactly one root in $[1,m+1]$. Moreover $p(\gamma)$ has at most one critical point $\gamma = -\frac{B}{A}$ in $[1,m+1]$. As a consequence on $[1,m+1]$ we have the following.
\begin{gather}
\displaystyle \int _1 ^{\gamma} p(t)tdt \geq \min (0, \int _1 ^{m+1} p(t)tdt) \ and \nonumber \\
\int _1 ^{m+1} p(t)tdt = LC + N, \ where \nonumber \\
L = \frac{m^2+2m}{2}-\frac{(m+1)^4-1}{4} \left [1+\frac{1}{m} - \frac{1}{m(m+1)^2}\right ] + \frac{(m+1)^5-1}{5m} \left [ 1- \frac{1}{(m+1)^2} \right ] \nonumber \\
N = -\frac{(m+1)^5-1}{5} \frac{2}{m} \left [1+\frac{1}{(m+1)^2} \right ] + \frac{1}{2} ((m+1)^4-1) \left [ 1+\frac{1}{m} + \frac{1}{m(m+1)^2} \right ].
\label{LC}
\end{gather}
If $C \leq 2$ then $LC+N>0$ which implies that $\int_ 1^{\gamma } p(t)tdt >0$.
\label{pgamma}
\end{lemma}
\begin{proof}
 Since $p(m+1) =-2$ and $p(1) =2$, $p$ has an odd number of roots (counted with multiplicity) in $[1,m+1]$. Now $p^{'} = \gamma (A(C) \gamma+B(C))$ which has at most one root in $[1,m+1]$. This implies that $p$ has exactly one root $\gamma_0$ in $[1,m+1]$. This also means that if there exists a smooth solution of \ref{initval} on $[1,m+1]$ satsifying $\phi(m+1)=0$ then $\phi>0$ on $(1,m+1)$.\\
\indent Notice that $\gamma \rightarrow \int _1 ^{\gamma}p(t)t dt$ assumes its minimum over $[1,m+1]$ on the boundary because its only critical point is a local maximum. An easy calculation shows that indeed $\displaystyle \int _1 ^{m+1} p(t)tdt = LC+N$ where $L$ and $N$ are as above. The following proves that indeed $L<0$ and $N>0$ for $m\geq 1$.
\begin{align}
L &= \frac{(m+1)^2-1}{2} - \frac{(m+1)^4-1}{4m} \left [m+1 - \frac{1}{(m+1)^2}\right ]+\frac{(m+1)^4}{5}  \left [ 1- \frac{1}{(m+1)^2} \right ]\nonumber \\
&+  \frac{(m+1)^4-1}{5m}\left [ 1- \frac{1}{(m+1)^2} \right ]   \nonumber \\
&= \frac{3}{10}(m+1)^2-\frac{1}{20}(m+1)^4-\frac{1}{4}-\frac{(m+1)^4-1}{20m} \left [1-\frac{1}{(m+1)^2} \right ] <0 \ \forall \ m\geq 1\nonumber \\
N &= -\frac{m(m+1)^4-1+(m+1)^4}{5} \frac{2}{m} \left [1+\frac{1}{(m+1)^2} \right ]+ \frac{(m+1)^4-1)}{2m} \left [ m+1 + \frac{1}{(m+1)^2} \right ] \nonumber \\
&= \frac{1}{10}(m+1)^4-\frac{1}{2}-\frac{2}{5}(m+1)^2+\frac{(m+1)^4-1}{10m} \left [1+\frac{1}{(m+1)^2} \right ]>0 \ \forall \ m\geq 1.
\label{LNpos}
\end{align}
Let $C=2-\delta$ where $\delta\geq 0$. Then 
\begin{align}
LC+N &=2L+N-\delta L >2L+N \nonumber \\
&= (m+1)^2-1+((m+1)^4-1)\frac{1}{m(m+1)^2}-\frac{4}{5}\frac{(m+1)^5-1}{m(m+1)^2} \nonumber \\
&= (m+1)^2-1 + ((m+1)^4-1)\frac{1}{m(m+1)^2}-\frac{4}{5}\frac{(m+1)^4-1}{m(m+1)^2}-\frac{4}{5}(m+1)^2 \nonumber \\
&=\frac{(m+1)^2}{5} -1 + \frac{1}{5}((m+1)^4-1)\frac{1}{m(m+1)^2}  \nonumber \\
&=\frac{(m+1)^2}{5}-\frac{4}{5}+\frac{m+1}{5}+\frac{1}{5(m+1)} +\frac{1}{5(m+1)^2} >\frac{2}{5}
\label{Cleqtwo} 
\end{align}
\end{proof}

We now conclude the proof of theorem \ref{partialode}. Given $m$, if $C$ is chosen so that 
\begin{gather}
\displaystyle \int _1 ^{m+1} p(\gamma) \gamma d\gamma \geq -2 +\epsilon, \nonumber \\
i.e., LC+N \geq -2+\epsilon
\label{integralcond}
\end{gather}
then 
\begin{gather}
v(\gamma) -v(1)= \displaystyle \int _1 ^{\gamma} 2\sqrt{2} \sqrt{v}+ \int_1 ^{\gamma} p(t) t dt \nonumber \\
\Rightarrow v(\gamma) > 2-2+\epsilon = \epsilon. 
\label{boundbelow}
\end{gather}
This implies that for $C$ satisfying \ref{integralcond} (in particular, by lemma \ref{pgamma} $C\leq 2$  satisfies \ref{integralcond} for all $m\geq 1$) we have a smooth solution to \ref{vsub}, hence to \ref{initval} on $[1,m+1]$. Now we have to somehow choose a $C$ so that $\phi(m+1)=0$, i.e., $v(m+1)=2(m+1)^2$. One possible strategy is to show that there is a $C$ satisfying \ref{integralcond} such that $v(m+1,C) <2(m+1)^2$ and likewise another $C$ for which $v(m+1,C)>2(m+1)^2$. Thus there will exist a $C$ so that $v(m+1,C)=2(m+1)^2$. \\
\indent If $C$ is very negative then $LC+N$ can be made as large as we want. Thus $v(m+1,C)>2+LC+N>2(m+1)^2$. This completes the proof of theorem \ref{partialode}
\end{proof}

\indent We proceed further to prove theorem \ref{mainrule}. As mentioned earlier, this reduces to choosing $C$ so that $LC+N\geq-2+\epsilon$ for some $\epsilon>0$ so that $v(m+1,C)<2(m+1)^2$. This is a tricky business. Here is where we use the assumption that $m=1$. For this we need to choose $\delta>0$ to be very small so that among other things  $C=2+\delta$ satisfies $\displaystyle LC+N =-\frac{33}{20}$, $A(C)>0$, and $B(C)<0$. Upon calculation we have the following.
\begin{align}
\frac{A}{3} &= \frac{\delta}{m} \left [1-\frac{1}{(m+1)^2} \right ] - \frac{4}{m(m+1)^2} \nonumber \\
\frac{B}{2} &= -\delta \left [1+\frac{1}{m}-\frac{1}{m(m+1)^2} \right ] + \frac{4}{m(m+1)^2}  \nonumber \\
\Rightarrow \ &\mathrm{if}  \ \delta > \frac{4}{(m+1)^2-1} \ then \ A>0 \ , \ B<0 \nonumber \\
LC+N &= \delta L + \frac{(m+1)^2}{5}-\frac{4}{5}+\frac{m+1}{5}+\frac{1}{5(m+1)} +\frac{1}{5(m+1)^2} \nonumber \\
&= \delta \Bigg ( \frac{3}{10}((m+1)^2-1)-\frac{1}{20}((m+1)^4-1)-\frac{(m+1)^4-1}{20m} \left [1-\frac{1}{(m+1)^2} \right ] \Bigg ) \nonumber\\
&+ \frac{(m+1)^2}{5}-\frac{4}{5}+\frac{m+1}{5}+\frac{1}{5(m+1)} +\frac{1}{5(m+1)^2} \nonumber \\
&= \delta^{'} L \ , where \  \delta = \delta ^{'} + \frac{4}{(m+1)^2-1} \nonumber \\
\Rightarrow \frac{A}{3} &= \frac{\delta^{'}}{m}  \left [1-\frac{1}{(m+1)^2} \right ] \nonumber \\
\frac{B}{2} &= -\delta^{'} \frac{(m+1)^3-1}{m(m+1)^2} -\frac{4}{(m+1)^2-1} .
\label{delt}
\end{align}

Therefore $\delta^{'} = \frac{-33}{20L}$.\\

\indent We now prove that for $m=1$ and the chosen value of $C=2+\frac{4}{(m+1)^2-1} -\frac{1.5}{L} = \frac{22}{3}$ the solution $v$ satisfies $v^{'} >0$ on $[1,2]$. Before this we note that $A=9$ and $B=\frac{50}{3}$. If $\gamma_0$ is the root of $p(\gamma)$ on $[1,2]$ then on $[1,\gamma_0]$ we see that 
\begin{gather}
v^{'} \geq 2\sqrt{2} \sqrt{v} \nonumber \\
\Rightarrow (\sqrt{v})^{'} \geq \sqrt{2} \Rightarrow \sqrt{v}(\gamma_0) \geq \sqrt{2}+\sqrt{2} (\gamma_0-1) = \sqrt{2} \gamma_0. \nonumber
\end{gather}
Therefore, $v^{'} >0$ on $[1,\gamma_0]$. On the other hand, the root $\gamma_0$ in $[1,2]$ of the polynomial $p(\gamma) \gamma = 3\gamma ^4 -\frac{25}{3} \gamma ^3+\frac{22}{3} \gamma$ is clearly larger than $1.2$. Therefore $\sqrt{v(\gamma_0)} > 1.2\sqrt{2} $. Moreover, one can also see (by graphing for instance) that $p(\gamma)\gamma>-4.5$ on $[1,2]$. But $v^{'}(\gamma_0) = 2\sqrt{2}\sqrt{v(\gamma_0)} =4.8$ and hence when $\gamma > \gamma_0$ we see that $v^{'}(\gamma) >-4.5 + 4.8 =0.3$. This proves that $v^{'}>0$ on $[1,2]$.\\

\indent As a consequence, for $a,a+h \in [1,2]$ we see that
\begin{align*}
&2\sqrt{2v(a)}h+\displaystyle \int _a^{a+h} p(\gamma) \gamma d\gamma < v(a+h)-v(a)< 2\sqrt{2v(a+h)}h+\displaystyle \int _a^{a+h} p(\gamma)\gamma d\gamma 
\end{align*}
\begin{align}
\Rightarrow \ &2\sqrt{2v(a)}+\displaystyle \int _a^{a+h} p(\gamma) \gamma d\gamma \leq v(a+h) \nonumber \\
&\leq 4h^2 + v(a) +\displaystyle \int _a^{a+h} p(\gamma)\gamma d\gamma+2h\sqrt{4h^2+2\left (v(a)+\displaystyle \int _a^{a+h} p(\gamma)\gamma d\gamma \right)}.
\label{keyineq}
\end{align}  
Using inequality \ref{keyineq} twice with $h=\frac{1}{2}$ and $a=1$ we see that $v(2) \leq 7.5 < 2(1+1)^2 = 8$. This proves that for $m=1$ indeed there exists a $C$ so that $\phi(m+1)=0$ thus almost proving theorem \ref{mainrule}. The only thing left is to prove that there cannot exist any hcscK metrics.\\
\indent Indeed, if such a metric exists then there is a solution to \ref{initval} satisfying $\phi>0$ (and hence $v>2\gamma^2$) and $A=0$. In this case $B=-\frac{12}{(m+1)^2-1}$ and $C= 4+\frac{8}{(m+1)^2-1}$. This implies that $\displaystyle \int _1 ^{m+1} p(\gamma)\gamma d\gamma =2$. Therefore,
\begin{gather}
v(m+1) =4+\displaystyle \int _1^{m+1} 2\sqrt{2v} d\gamma \nonumber 
>4+4\int _1 ^{m+1} \gamma d\gamma = 4+2((m+1)^2-1) > 2(m+1)^2. \nonumber
\end{gather}  
This is a contradiction.
\section{Perturbation results}
In this section we prove theorem \ref{mainperturb}. Let $(X,\omega)$ be a compact K\"ahler surface, $\tilde{\omega}$ be any closed real $(1,1)$-form, and let $\mathcal{B}_1$ and $\mathcal{B}_2$ be spaces of $C^{4,\alpha}$ functions on $X$ with zero average and $C^{0,\alpha}$ $(2,2)$-forms on $X$ with zero average respectively. Denote by $U$ an open subset of $\mathbb{R}\times\mathcal{B}_1$ consisting of  $(t,\phi) \in \mathbb{R}\times \mathcal{B}_1$ such that $\omega+ t\tilde{\omega}+\spbp \phi >0$. Consider the following map $L: U \rightarrow \mathcal{B}_2$.
\begin{gather}
L(t,\phi) = c_2 (\omega+t\tilde{\omega}+\spbp \phi) - \frac{\displaystyle \int _X c_2}{\displaystyle \int _X (\omega +t\tilde{\omega})^2} (\omega+t\tilde{\omega}+\spbp \phi)^2  
\label{defiL}
\end{gather}
Clearly $L^{-1}(0)$ consists of hcscK metrics in the K\"ahler class $[\omega +t\tilde{\omega}]$. Assume now that $\omega$ is an hcscK metric satisfying $c_2(\omega) = \frac{\lambda}{2(2\pi)^2} \omega^2$. In order to apply the implicit function theorem on Banach manifolds, we will linearise $L$ with respect to $\phi$ at $\phi=0, t=0$. Indeed,
\begin{gather}
DL_{t=0,\phi=0}(\psi) = \frac{d}{ds}\vert_{s=0} c_2(\omega+s\spbp \psi) - \frac{\lambda}{(2\pi)^2} \omega \spbp \psi. 
\label{linear}
\end{gather}
We have a small lemma in the making.
\begin{lemma}
The linearisation $DL$ given by equation \ref{linear} is uniformly elliptic in $\psi$ if the holomorphic sectional curvature has a definite sign throughout $X$.
\label{elliptic}
\end{lemma}
\begin{proof}
Let $P(A,B)$ be the polarisation of the determinant of $2\times 2$ matrices $A$ and $B$, i.e., if $A$ and $B$ are thought of as $2$-forms then $P(A,B) =\frac{A \wedge B}{2}$. Proposition (6) of \cite{Don} states (in this special case) that there exists a smoothly varying family of Bott-Chern forms $bc_2(h,k)$ such that the following holds.
\begin{gather}
c_2(\omega+s\spbp \psi)-c_2(\omega) = -\frac{\spbp}{2\pi} bc_2 (\omega+s\pbp \psi, \omega), \ \mathrm{and} \ \nonumber \\
\frac{d}{ds} bc_2 (\omega+s\spbp \psi, \omega) = -2\sqrt{-1}P\Bigg( h^{-1} \frac{dh}{ds},\frac{\sqrt{-1}}{2\pi}\Theta_{h}\Bigg), \nonumber 
\end{gather}
where $h = \omega+s\spbp \psi$ and $\Theta_h$ is the curvature of $h$. Using this result, we may compute the linearisation of $L$ to be the following.
\begin{gather}
DL_{\phi=0, t=0}(\psi) = -2\frac{1}{(2\pi)^2} \bar{\partial} \partial P\left(\omega^{i\bar{k}}\sqrt{-1}\frac{\partial^2\psi}{dz^{j}d\bar{z}^k}, \Theta \right) - \frac{\lambda}{(2\pi)^2} \omega \spbp \psi,
\label{linearsimple}
\end{gather}
where $\Theta$ is the curvature of $\omega$. In order to find the principal symbol let us choose coordinates such that $\omega =\sqrt{-1}\sum dz^{i}\wedge d\bar{z}^i$. Replacing $\partial$ by a covector $\xi$ we see that the prinicipal symbol is $\frac{2}{(2\pi)^2}\Theta (\xi \wedge \bar{\xi}, \xi\wedge \bar{\xi}) dz^1 \wedge d\bar{z}^1 \wedge dz^2 \wedge d\bar{z}^2$ which is just the holomorphic sectional curvature. Hence, it having a definite sign (along with compactness of $X$) implies uniform ellipticity.
\end{proof}
From now onwards we will specialise to $(X,\omega)$ being one of the symmetric surfaces in the statement of theorem \ref{mainperturb}. In the cases considered in theorem \ref{mainperturb} the holomorphic sectional curvature has a sign and hence by lemma \ref{elliptic} equation \ref{linearsimple} is uniformly elliptic of the fourth order. By the Fredholm alternative, it is surjective if and only if the kernel of its formal adjoint operator is trivial. It is easy to see that $DL_{\phi=0, t=0}$ is symmetric on the space of smooth functions. The following lemma implies that $DL$ is an isomorphism.
\begin{lemma}
If $(X,\omega)$ is a K\"ahler surface in theorem \ref{mainperturb}, then the kernel of $DL_{\phi=0}$ is trivial.
\end{lemma}
\begin{proof}
 Suppose $DL(\psi) =0$, multiplying and integrating by parts we see that (implicitly writing in terms of normal coordinates)
\begin{gather}
2\displaystyle \int_X \pbp \psi \wedge P(\psi_{i\bar{j}},\Theta)+\lambda \int_X \sqrt{-1}\partial \psi \wedge \delb \psi \wedge \omega  = 0.
\label{intparts}
\end{gather}
For the surfaces in question it is clear that $\lambda >0$. Suppose we choose normal coordinates such that $\psi_{i\bar{j}} = diag(\mu_1, \mu_2)$, then
\begin{align}
\displaystyle \pbp \psi \wedge P(\psi_{i\bar{j}},\Theta) &= \sum \mu_i dz^i \wedge d\bar{z}^i \wedge \frac{\mu_1 \Theta _{2\bar{2}} + \mu_2 \Theta_{1\bar{1}}}{2} \nonumber \\
&= -(\mu_1 ^2 \Theta _{2\bar{2}2\bar{2}} + 2\mu_1 \mu_2 \Theta _{1\bar{1}2\bar{2}} + \mu_2 ^2 \Theta _{1\bar{1}1\bar{1}}) \sqrt{-1} ^2 dz^1 \wedge d\bar{z}^1 \wedge dz^2 \wedge d\bar{z}^2 \nonumber \\
&= -\Theta (\sum \mu_i \frac{\partial}{\partial z_i} \wedge \frac{\partial} {\partial \bar{z_i}},\sum \mu_i \frac{\partial}{\partial z_i} \wedge \frac{\partial} {\partial \bar{z_i}})  \sqrt{-1} ^2 dz^1 \wedge d\bar{z}^1 \wedge dz^2 \wedge d\bar{z}^2. 
\label{normalsimple}
\end{align}
 For $X= \mathbb{D}^2/ \Gamma$ and $X=\mathbb{D}^1/\Gamma_1 \times \mathbb{D}^1/ \Gamma_2$ equipped with their ``canonical" metrics, the curvature operator is nonpositive. Hence $\nabla \psi=0$ and thus $\psi=0$.
\end{proof}
By the Fredholm alternative and the Schauder estimates $DL$ is indeed an isomorphism. Therefore by the implicit function theorem on Banach spaces, for small $t$ there exists a unique hcscK metric in a $C^{2,\alpha}$ neighbourhood of $\omega$ in the class $[\omega +t\tilde{\omega}]$ depending smoothly on $t$. In particular, for some ball quotients we can choose $\tilde{\omega}$ to be in a  cohomology class that is not a multiple of the first Chern class and therefore get a non K\"ahler-Einstein example of an hcscK metric.
\section{Bando-Futaki invariants of projective hypersurfaces}
Let $M$ be a compact K\"ahler manifold. The Bando-Futaki invariants associated to a given K\"ahler class $\omega$ and a given holomorphic vector field $Y$ (henceforth denoted as $\mathcal{F} _k (Y,\omega)$) are obstructions to the harmonicity of the Chern forms $c_k$ of the holomorphic tangent bundle. By Hodge theory there exists a smooth function $g_k$ such that $$c_k - H(c_k) = \frac{\sqrt{-1}}{2\pi} \partial \bar{\partial} g_k$$ 
where $H(c_k)$ is the harmonic projection of $c_k$. The Bando-Futaki invariants are defined as
$$\cF _k (Y, \omega) = \displaystyle \int _{M} L_{Y} g_k \wedge \omega ^{n-k+1}.$$ where $L_Y$ is the lie derivative with respect to $Y$. \\
\indent The fact that these functions are actually invariants of the K\"ahler class was proven by Bando \cite{Ban}. In Liu's paper \cite{Liu} these invariants were computed for a smooth, degree $d$ hypersurface $M$ of $\mathbb{CP}^n$ for the Fubini-Study K\"ahler class. Liu speculated that an ``abstraction" of the procedure used is desirable (in order to compute the same for complete intersections). We simplify some aspects of Liu's proof (whilst following the same basic strategy) thus providing a possible abstraction of that method.\\
\indent An important tool in our calculations is the following linear algebra lemma which has proven to be quite useful in the calculation of characteristic forms \cite{PiTa}.
\begin{lemma} \label{Algebra} Let $A$ be a matrix over $\C$ or over a commutative algebra $\mathcal{A}$ over $\mathbb{C}$, where in the latter case all its matrix elements are nilpotent. Suppose that $A^2=aA$ for some
$a\in\mathcal{A}$, and that $1-\lambda a$ is invertible for all $\lambda$ in some domain $D\subset\C$ containing $0$. Then for such $\lambda$ we have
$$ (I-\lambda A)^{-1} = I+\frac{\lambda}{1-\lambda a} A,$$ and
$$\det(I-\lambda A) = \exp\left\{\frac{\Tr A}{a}\log(1-\lambda a)\right\}.$$ 
In particular, if $\alpha_{i}, \beta_{i}$, $i=1,\dots,k$, are odd elements in some  graded-commutative algebra over $\C$ (e.g., the algebra of complex differential forms on $X$), and $A_{ij}=\alpha_{i}\beta_{j}$, then  $A^2=aA$ where $a=-\Tr A=-\sum_{i=1}^{k}\alpha_{i}\beta_{i}$, and
$$\det(I-\lambda A)=\frac{1}{1-\lambda a}.$$
\end{lemma}
\begin{proof}  For $\lambda\in D$ we have
\begin{equation*} 
(I-\lambda A)^{-1}=I +\frac{\lambda}{1-\lambda a}\,A.
\end{equation*}
To prove the formula for the determinant, we use the identity
$$\frac{d}{d\lambda}\log\det(I-\lambda A)=-\Tr \left\{A(I-\lambda A)^{-1}\right\},\quad\lambda\in D.$$
It is well-known for matrices over $\C$ (and easily proved using the Jordan canonical form), and for matrices with nilpotent entries it easily follows from the definition of the determinant. Using formula for the inverse, we obtain
$$\frac{d}{d\lambda}\log\det(I-\lambda A)=-\frac{\Tr A}{1-\lambda a}=\frac{d}{d\lambda}\frac{\Tr A}{a}\log (1-\lambda a),$$
and integrating from $0$ to $\lambda$ using $\det I=1$ gives the result.
\end{proof}
 Our starting point of Liu's formula is the expression for the curvature of the induced metric on the hypersurface $M$ defined by $F(Z_0,Z_1,\ldots,Z_n)=0$ where $F$ is a homogeneous polynomial with non-zero gradient. On the set where $Z_0 \neq 0$, define the complex coordinates $z_i = \frac{Z_i}{Z_0}$ for $i\geq 1$. Defining $f = F[1,\frac{Z_1}{Z_0}, \ldots,\frac{Z_n}{Z_0}]$, if $\frac{\partial f}{\partial z_1} \neq 0$, then by the implicit function theorem $z_1$ is a holomorphic function of the other coordinates. Let $a_i = \frac{\partial z_1}{\partial z_i}$, $\widetilde{g}$ be the metric on $M$ induced by the Fubini-study metric $\omega _{FS} = \frac{\sqrt{-1}}{2\pi}\sum _{i,j} \left ( \frac{\delta _{ij}}{1+\vert z \vert^2}-\frac{z_i \bar{z}_j}{(1+\vert z \vert^2)^2} \right )dz_i \wedge d\bar{z}_j$, $F_k = \frac{\partial F}{\partial Z_k}$, and $\rho =  \frac{\sum _{k=0} ^{k=n}\vert F_k \vert ^2}{(1+\vert z \vert^2)\vert F_1 \vert ^2}$. It is easy to see that\footnote{Either using \cite{Liu} or by noting that one may compute the inverse of the metric and hence the curvature by using lemma \ref{Algebra}}     
\begin{eqnarray}
\widetilde{g_{\mu \nu}} &=& \frac{\delta _{\mu \nu} + a_{\mu} \bar{a}_{\nu}}{1+\vert z \vert^2} - \frac{(\bar{z}_{\mu} + \bar{z}_1 a_{\mu})(z_{\nu} + z_1 \bar{a}_{\nu})}{(1+\vert z \vert^2)^2} \nonumber \\
\Theta _{\mu \nu} &=&  \widetilde{g_{ij}} dz_{i} \wedge \bar{dz_{j}} \delta_{\mu \nu} - \widetilde{g_{\mu j}} \bar{dz_j} \wedge dz_{\nu} - \frac{1}{\rho} ( \frac{\partial a_{\mu}}{\partial z_i}dz_i \wedge \frac{\partial \bar{a}_s}{\partial \bar{z}_j} \widetilde{g^{\nu s}}\bar{dz_j}) \nonumber 
\end{eqnarray}
Now, we shall state and prove lemma $2.3$ of \cite{Liu}
\begin{lemma}
The $q$th Chern form of the degree $d$ hypersurface $M$ is 
\begin{gather}
c_{q}(\Theta) = \sum _{k=0}^{q} \alpha _{qk} (\frac{\sqrt{-1}}{2\pi}\omega) ^k \wedge (\frac{\sqrt{-1}}{2\pi} \partial \bar{\partial} \xi )^{q-k} \nonumber
\end{gather}
where  $\xi = \log (\frac{\sum _{k=0}^{n}\vert F_k \vert^2}{(\sum _{k=0}^{k=n} \vert Z_k \vert ^2)^{d-1}}) $, and
\begin{eqnarray}
\alpha _{00} &=& 1 \nonumber \\
\alpha _{qq} &=& \binom {n+1}{q} -d\alpha _{(q-1)(q-1)} \nonumber \\
\alpha _{q(q-k)} &=& -[d\alpha _{(q-1)(q-k-1)}+\alpha _{(q-1)(q-k)}] \ for \ k= 1,\ldots, q-1 \nonumber \\
\alpha _{q0} &=& (-1)^q \nonumber 
\end{eqnarray}
where $q$ ranges from $1$ to $n-1$.
\label{Chernforms}
\end{lemma}
\begin{proof}
 We use lemma \ref{Algebra} quite often in what follows. For the sake of brevity we denote $a \wedge b$ by $ab$ from now onwards.
\begin{align}
\Theta _{ij} &= \omega \delta _{ij} + v_i w_j + \alpha _i \beta _j \nonumber \\
\det(I+t\Theta) &=  \det(\delta _{ij}(1+t\omega)+t(v_i w_j + \alpha _i \beta _j) \nonumber \\
&= (1+t\omega)^{n-1} \det(\delta_{ij}+\frac{t}{1+t\omega}(v _i  w _j + \alpha _i  \beta _j)) \nonumber \\
= (1+t\omega)^{n-1} \det(\delta_{ij}+\frac{t}{1+t\omega} v _i w _j) &\times 
 \det(\delta_{ij}+ (\delta_{ab}+\frac{t}{1+t\omega}v_a w_b)^{-1}\frac{t}{1+t\omega} \alpha _i \beta _j ) \nonumber \\
&= (1+t\omega)^{n-1} \det(\delta_{ij}+\lambda v_i  w_j )\det(I+\lambda A) \nonumber 
\end{align} 
where  $\omega = \widetilde{g}_{\mu \nu} dz_{\mu} \wedge d\bar{z}_{\nu}$, $v_{\mu} = -\widetilde{g}_{\mu j} d\bar{z}_j$, $w_{\nu} = dz_{\nu} $, $\alpha _{\mu} = -\frac{1}{\rho}\frac{\partial a_{\mu}}{\partial z_i} dz_i $, $\beta _{\nu} = \frac{\partial \bar{a}_s}{\partial \bar{z}_j} \tilde{g}^{\nu s} d\bar{z}_j$, $\lambda = \frac{t}{1+t\omega} $, $u_i =\frac{t}{(1+t\omega)+t w_j \wedge v_j}v_i$, and $A_{ij}= (\delta_{ab}+\frac{t}{1+t\omega}v_a w_b)^{-1} \alpha _i \beta _j $. \\
Now notice that $A^2=(\beta _i  \alpha _i - \beta _j  u_j  w_k  \alpha _k) A = -tr(A) A$.  Using lemma \ref{Algebra}, we see that 
\begin{eqnarray}
\det(1+t \Theta) &=& (1+t\omega)^{n+1}\frac{1}{1+t\omega + \frac{t}{\rho}\frac{\partial a_{\mu}}{\partial z_i} dz_i \wedge \frac{\partial \bar{a}_s}{\partial \bar{z}_j} \tilde{g}^{\mu s} d\bar{z}_j} \nonumber 
\end{eqnarray}
From \cite{Liu} we see that $\frac{1}{\rho}\frac{\partial a_{\mu}}{\partial z_i} dz_i \wedge \frac{\partial \bar{a}_s}{\partial \bar{z}_j} \tilde{g}^{\mu s} d\bar{z}_j = (d-1) \omega + \partial \bar{\partial}\xi$. Hence, we see that the coefficient of $t^k$ in the above expression is 
\begin{eqnarray}
c_a({\Theta}) &=& \sum _b \sum _l \binom {n+1}{b} d^l \omega ^{b+l} (-1)^{a-b}\binom{a-b}{l} (\partial \bar{\partial} \xi)^{a-b-l} \nonumber \\
&=& \sum _{k=0} ^{a} \sum _{b=0} ^k \binom{n+1}{b} d^{k-b} \omega ^k (-1) ^{a-b} \binom{a-b}{k-b} (\partial \bar{\partial}\xi)^{a-k} \nonumber 
\end{eqnarray}
From this, the lemma follows.
\end{proof}
\indent At this juncture we may compute the Bando-Futaki invariants using a generating series version of Liu's approach. Our basic strategy of proof is the same as Liu's, in that we shall not compute the invariant directly. Instead, we  observe that $i_{Y} (c(\Theta))-i_Y(H(c(\Theta))) - \bar{\partial}(i_Y (\partial f)) =0$ (where $c$ and $f$ are the Chern and the Futaki polynomials respectively). This shall be rewritten as $\bar{\partial} \eta = 0$ and after that we shall find the harmonic part of $\eta$ to finally compute the integral.  In the course of the proof we use lemma \ref{Algebra} repeatedly. \\ \textbf{}\\
\emph{Proof of theorem \ref{Liuthm}}:\\
\indent First, we recall that $\det (I+t\Theta) = \frac{(1+t\omega )^{n+1}}{1+t(\omega d + \partial \bar{\partial} \xi )}$. The harmonic part of the same maybe obtained by putting $\xi=0$. Hence,
\begin{eqnarray}
c - Hc &=& (1+t\omega)^{n+1} (\frac{1}{1+t(\omega d + \pbp \xi)} - \frac{1}{1+t\omega d} ) \nonumber \\
&=& -t\pbp (\frac{\xi (1+t\omega)^{n+1}}{(1+t(\omega d + \pbp \xi))(1+t\omega d)}) \nonumber \\
&=& \pbp f \nonumber 
\end{eqnarray}
In what follows, $\theta$ is the ``Hamiltonian" function \cite{Liu} such that $i_Y \omega = -\bar{\partial} \theta$. We shall use the fact that $i_Y$ is a derivation (and hence the quotient and the product rules for derivatives maybe used when interpreted suitably).  \\
\begin{align}
i_Y (Hc) &= \frac{(n+1)(1+t\omega)^nti_Y (\omega) (1+t\omega d) - ti_Y (\omega) d (1+t\omega)^{n+1}}{(1+t\omega d)^2} \nonumber \\
&= \bar{\partial}(\frac{t(1+t\omega)^n \theta (d-(n+1)-nt\omega d)}{(1+t\omega d)^2}) \nonumber \\
&= \bar{\partial} \alpha _2 \nonumber \\
(I+t\Theta)^{-1} &= \frac{1}{1+t\omega}(\delta_{ij}+\frac{t}{1+t\omega}(v_i w_j + \alpha _i  \beta _j))^{-1}\nonumber \\
&= \frac{1}{1+t\omega}(\delta_{ij}+\frac{t}{1+t\omega}v_i w_j)^{-1} (\delta_{ij} + \frac{t}{1+t\omega}((\delta_{ab}+\frac{t}{1+t\omega}v_a w_b)^{-1})_{ik}\alpha _k \beta _j)^{-1}\nonumber 
\end{align}
Using lemma \ref{Algebra}, and noticing that $w_k v_k = -\omega$ and $\beta_k \alpha _k = (d-1)\omega + \partial \bar{\partial} \xi$ we see that
\begin{gather}
((I+t\Theta)^{-1})_{ab} = \frac{1}{1+t\omega} (\delta_{ac}-t v_{a} w_{c})(\delta_{cb}-\frac{t(\alpha _c \beta _b -t v_c  w_k  \alpha _k \beta _b) }{1+t\omega d + t \partial \bar{\partial} \xi - t^2  \beta _k v_k   w_l \alpha _l }) \nonumber 
\end{gather}
\begin{align}
i_Y (c) &= \det (I+t\Theta)\mathrm{tr}(ti_Y (\Theta) (I+t\Theta)^{-1}) \nonumber \\
&= -t \bar{\partial} ( \det (I+t\Theta) \mathrm{tr} (\nabla Y (I+t\Theta)^{-1}))  
\label{Temp} 
\end{align}
We use the following equations from \cite{Liu}
\begin{eqnarray}
(\nabla Y)^{l}_{k} &=& -\widetilde{g}^{lj}\partial_{k}\bar{\partial}_j \theta \nonumber \\
\Phi &=& -\frac{1}{\rho}Y^{l}_{;k} \frac{\partial a_l}{\partial z^p}\frac{\partial \bar{a}_s}{\partial \bar{z}^q}\widetilde{g}^{k\bar{s}}dz^p \wedge d\bar{z}^q \nonumber \\
&=& \mathrm{div}(Y) ((d-1)\omega + \partial \bar{\partial}\xi) -\partial \bar{\partial}\theta  + \partial \bar{\partial}\Delta \theta \nonumber \\
&-& (n+1)\theta ((d-1)\omega + \partial \bar{\partial}\xi) \nonumber 
\end{eqnarray}
Upon simplification of equation \ref{Temp} (recall that since $a_i = \frac{\partial z_1}{\partial z_i}$, $w_k \alpha _k =0$) we get
\begin{align}
i_Y (c) &= -t \bar{\partial} (\frac{(1+t\omega)^n}{1+t\omega d + t\pbp \xi} (\mathrm{div}(Y)+t\bar{\partial}\partial \theta - \frac{t\Phi}{1+t\omega d + t \partial \bar{\partial} \xi} )) \nonumber \\
&= \bar{\partial} \alpha _1 \nonumber \\
i_Y (\partial f) 
 &=  t\frac{-Y(\xi)(1+t\omega)^{n+1}}{(1+t(\omega d + \pbp \xi))(1+t\omega d)} \nonumber \\
 &- \frac{t^2(1+t\omega)^{n+1}\partial \xi}{(1+t(\omega d + \pbp \xi))^2(1+t\omega d)^2}\bar{\partial}[\theta((n+1-d+nt\omega d)(1+t(\omega d + \pbp \xi)) \nonumber \\
&-(1+t\omega)(1+t\omega d)d) + Y(\xi) ((1+t\omega)(1+t\omega d) )]\nonumber 
\end{align} 
It is easy to see that for an appropriate form $\gamma$, we have,
\begin{gather}
\alpha _1 - \alpha _2 - i_Y (\partial f) 
= \bar{\partial} \gamma 
+ \frac{t(1+t\omega)^n}{(1+t\omega d)^2}\theta (nt\omega d + n+1-d)  -i_Y (\partial f) \nonumber \\
 -t \frac{(1+t\omega)^n}{1+t(\omega d + \pbp \xi)^2}(\mathrm{div}(Y)(1+t\omega) + (n+1)t\theta ((d-1)\omega+\pbp \xi)) \nonumber 
\end{gather}
We shall use this identity \cite{Lu},
\begin{eqnarray}
\mathrm{div}(Y)-Y(\xi)-(n-d+1)\theta &=& -\kappa \nonumber  
\end{eqnarray}
Replacing $\mathrm{div}(Y)$ by the above identity and simplifying we have,
\begin{align}
&\alpha _1 - \alpha _2 - i_Y (\partial f) = \bar{\partial}\gamma + t^2 \frac{\kappa (1+t\omega)^{n+1}}{(1+t(\omega d + \pbp \xi))^2} \nonumber \\
 &-\bar{\partial} \Bigg( \frac{t^2(1+t\omega)^{n+1}\partial \xi}{(1+t(\omega d +\pbp \xi))^2(1+t\omega d)} \nonumber \\
&\times\Big[\frac{\theta ((n+1-d+nt\omega d)(1+t(\omega d +\pbp \xi))-(1+t\omega)(1+t\omega d)d)}{1+t\omega d} 
 + (1+t\omega)Y(\xi)\Big]\Bigg) \nonumber
\end{align}
Thus, the harmonic part is $ t^2 \frac{\kappa (1+t\omega)^{n+1}}{(1+t(\omega d + \pbp \xi))^2}$. Notice that (the integral of a non-top form is defined to be zero)
\begin{align}
\int _M L_Y f \wedge \frac{1}{1-\omega} &= \int _M (di_Y + i_Y \partial) f \wedge \frac{1}{1-\omega} \nonumber \\
&= \int _M (\alpha_1 - \alpha_2 - t^2 \frac{\kappa (1+t\omega)^{n+1}}{(1+t(\omega d + \pbp \xi))^2}) \wedge \frac{1}{1-\omega}  \nonumber 
\end{align}
\begin{align}
= \int _M (\alpha_1 - \frac{t(1+t\omega)^n \theta (d-(n+1)-nt\omega d)}{(1+t\omega d)^2} - t^2 \frac{\kappa (1+t\omega)^{n+1}}{(1+t\omega d)^2}) \wedge \frac{1}{1-\omega} \nonumber 
\end{align}
where Stokes' theorem was used to deduce that $\int _M di_Y f \wedge \frac{1}{1-\omega} = 0$, to replace $1+t(\omega d + \pbp \xi)$ by $1+t\omega d$ and to ignore the integral of the anharmonic part of $\alpha_1 - \alpha _2 - i_Y(\partial f)$. \\
\indent From lemma $2.6$ of \cite{Liu}, it follows that $\int _M \frac{\alpha_1}{1-\omega} = 0$. After replacing $t$ by $\frac{\sqrt{-1}}{2\pi}$, one may easily compute the integral using the facts that $\int _M \theta \omega ^{n-1} = \frac{\kappa}{n}$ and $\int \omega ^{n-1} = d$. This completes the proof.  \qed

\end{document}